\documentclass[11pt]{amsart}
\usepackage{amsmath,amsthm,amsfonts,amssymb,amscd,tikz}

\newcommand{\f}{\operatorname{f}}
\DeclareMathOperator\coker{coker}
\newcommand{\C}{\mathbb{C}}
\newcommand{\F}{\mathbb{F}}

\newcommand{\cE}{\mathcal{E}}
\newcommand{\A}{\mathbb{A}}
\renewcommand{\L}{\mathbb{L}}

\newcommand{\bbP}{\mathbb{P}}
\newcommand{\Z}{\mathbb{Z}}
\newcommand{\Q}{\mathbb{Q}}

\newcommand{\cA}{{\mathcal A}}

\newcommand{\cI}{{\mathcal I}}

\newcommand{\Spec}{\operatorname{Spec}\,}
\newcommand{\Span}{\operatorname{Span}}
\newcommand{\Hom}{\operatorname{Hom}}

\newcommand{\Gal}{\mathrm{Gal}}
\newcommand{\PGL}{\mathrm{PGL}}

\newcommand{\GL}{\mathrm{GL}}
\newcommand{\SL}{\mathrm{SL}}
\newcommand{\Tr}{\mathrm{Tr}}
\newcommand{\Sym}{\mathrm{Sym}}

\newcommand{\End}{\mathrm{End}}

\newcommand{\Var}{\mathrm{Var}}

\newcommand{\Res}{\mathrm{Res}}

\newcommand{\K}{K_l^{\mathrm{sp}}}

\newcommand{\one}{\bar\F_\ell}
\newcommand{\indec}{\mathbf{W}}
\newcommand{\mtot}{\mu_\ell^{\mathrm{}}}

\newcommand{\ntot}{\nu_\ell^{\mathrm{}}}

\newcommand{\Cor}{\operatorname{Cor}}
\newcommand{\Chow}{\operatorname{Chow}}

\newcommand{\et}{}

\renewcommand{\ell}{l}
\newcommand{\hmod}{\hbox{-$\mathrm{mod}$}}
\newcommand{\hvect}{\hbox{-$\mathrm{vect}$}}
\newcommand{\cycl}{\operatorname{cl}}
\newcommand{\bbC}{{\mathbb C}}

\newcommand{\bbZ}{{\mathbb Z}}

\renewcommand{\bbP}{{\mathbb P}}

\newcommand{\bbQ}{{\mathbb Q}}

\newcommand{\isom}{\xrightarrow{\sim}}

\newtheorem{thm}{Theorem}[section]

\newtheorem{example}[thm]{Example}

\newtheorem{cor}[thm]{Corollary}

\newtheorem{prop}[thm]{Proposition}
\newtheorem{lemma}[thm]{Lemma}

\newtheorem{lem}[thm]{Lemma}
\newtheorem{defn}[thm]{Definition}

\newtheorem*{claim*}{Claim}
\numberwithin{equation}{section}

\begin{document}

\title{Irrationality of Motivic Zeta Functions}

\author{Michael J. Larsen}
\address{Department of Mathematics,
Indiana University,
Bloomington, IN
47405,
U.S.A.}

\author{Valery A. Lunts}
\address{Department of Mathematics,
Indiana University,
Bloomington, IN
47405,
U.S.A.}

\thanks{ML was partially supported by NSF grant DMS-1702152.}

\begin{abstract}
Let $K_0(\Var_{\Q})[1/\L]$ denote the Grothendieck ring of $\Q$-varieties with the Lefschetz class inverted.
We show that there exists a K3 surface $X$ over $\Q$ such that the motivic zeta function $\zeta_X(t) := \sum_n [\Sym^n X]t^n$
regarded as an element in $K_0(\Var_{\Q})[1/\L][[t]]$ is not a rational function in $t$, thus disproving a conjecture of
Denef and Loeser.
\end{abstract}

\maketitle

\section{Introduction}

Let $k$ be a field.  We denote by $K_0(\Var_k)$ the Grothendieck group of $k$-varieties, i.e., the free abelian group generated by isomorphism classes of $k$-varieties modulo the cutting-and-pasting relations $[X] = [Y] + [X\setminus Y]$ for all pairs $(X,Y)$ consisting of a variety $X$ and a closed subvariety $Y$.  It is endowed with a commutative ring structure characterized by $[X]\,[Y] = [X\times Y]$.  (Note that we use \emph{variety} to mean reduced separated scheme of finite type over $k$, but the Grothendieck ring would not be changed if we allowed non-reduced schemes or non-separated schemes, or limited ourselves to affine schemes.)

Following Kapranov \cite{Ka}, we define the \emph{motivic zeta function} 
$$\zeta_X(t) :=\sum_{n=0}^\infty [\Sym^n X]\,t^n\in K_0(\Var_k)[[t]],$$
where $\Sym^n X$ is the symmetric $n$th power $X^n/\Sigma_n$.

By a \emph{motivic measure}, we mean a homomorphism $\mu\colon K_0(\Var_k)\to A$, where $A$ is a commutative ring.
We write $\mu(\zeta_X(t))$ for the image of $\zeta_X(t)$ in $A[[t]]$.
If $k$ is a finite field, $\mu\colon [X]\mapsto |X(k)|$ defines a motivic measure with values in $\Z$.  The image 
$\mu(\zeta_X(t))\in \Z[[t]]$ is the usual zeta function of $X$ and therefore rational as a function of $t$
by Dwork's theorem \cite{Dw}.
Kapranov asked \cite[1.3.5]{Ka}
whether this rationality holds for the motivic zeta function itself.  He proved that this is so when $X$ is a curve with at least one $k$-point, even if $k$ is not a finite field.
(Since $K_0(\Var_k)$ is not an integral domain \cite{Po}, there is a question exactly what this means, which we settle for the purposes of this paper by saying that $\zeta_X(t)$ rational means that there exists a polynomial $B(t) = 1+b_1t+\cdots + b_nt^n$
such that $B(t)\zeta_X(t)\in K_0(\Var_k)[t]$.)

In \cite{LL1}, we proved that in general $\zeta_X(t)$ is \emph{not} rational when $X$ is a surface.  
This does not quite finish the question, since for many purposes (especially motivic integration), the natural object to consider is not $K_0(\Var_k)$ but $K_0(\Var_k)[1/\L]$, where
$\L := [\A^1]$.  It is known  \cite{Bo} that $\L$ is a zero-divisor; see, also, \cite{Za}, for an analysis of the annihilator of $\L$.  One might still hope, therefore, that 
$\zeta_X(t)$ may be rational as a power series over $K_0(\Var_k)[1/\L]$.  No variant of the method of \cite{LL1} can possibly test this, since the motivic measures constructed in
that paper are birationally invariant and therefore vanish on $\L$.  This made possible the conjecture of Denef and Loeser 
\cite[Conjecture 7.5.1]{DL} predicting that $\zeta_X(t)$ should satisfy this weaker rationality condition.
In this paper, we show that in general it does not.

To explain our strategy, we begin by discussing certain motivic measures which \emph{cannot} detect the irrationality of zeta functions.
A reference for the following discussion is \cite{LL2}.
We endow $K_0(\Var_k)$ with the  $\lambda$-structure
in which the $[X]\to [\Sym^n X]$ operations play the role of symmetric powers; in other words, $\lambda^n([X])$ is defined to be the $t^n$ coefficient of $\zeta_X(t)^{-1}$.  
If $A$ is a finite $\lambda$-ring (in the sense that every element $a\in A$ 
can be written $a = b-c$ where $\lambda^n b = \lambda^n c=0$ for $n$ sufficiently large), then every $\lambda$-homomorphism $\mu\colon K_0(\Var_k)\to A$ is a motivic measure for which $\mu(\zeta_X(t))$ is rational for all $X/k$.

Here is an example.  Let $K(G_k,\Q_\ell)$ denote the Grothendieck ring of (virtual) finite-dimensional continuous representations of $G_k$, where, as usual, $0\to V_1\to V_2\to V_3\to 0$ implies $[V_2] = [V_1] + [V_3]$.  Then $K(G_k,\Q_\ell)$ is a $\lambda$-ring (even a special $\lambda$-ring), and
$$[X]\mapsto \sum_{i=0}^{2\dim X} (-1)^i [H^i(\bar X,\Q_\ell)],$$
where $H^i(\bar X,\Q_\ell)$ denotes the $i$th $\ell$-adic \'etale cohomology group of $\bar X$ as $G_k$-representation,
defines a ring homomorphism $\mu$.
It is a consequence of the K\"unneth formula and the isomorphism 
$$H^i(\Sym^n \bar X,\Q_\ell)\isom H^i(\bar X^n,\Q_\ell)^{\Sigma_n}$$
that $\mu$ is a $\lambda$-homomorphism.  Thus $\mu(\zeta_X(t))$ is rational in $t$ for all $X$, where the degree of numerator and denominator depend only on the dimension of the cohomology of $\bar X$.

In particular, if $X$ is a K3 surface, then
$\mu(\zeta_X(t))^{-1}$ is a polynomial of degree $24$, the product of a degree $22$ polynomial corresponding to the $H^2$-term  and the factors $(1-t)(1-\mu(\L)^2t)$, corresponding to the $H^0$ and $H^4$ terms.  We consider K3 surfaces of Picard number 18, in which the $H^2$ factor further decomposes 
$(1-\mu(\L)t)^{18}\Lambda(t)$.

We modify this construction in three ways.  First, we consider coefficients in $\one$ instead of $\Q_\ell$.  Second, we use a modified Grothendieck ring 
$\K(G_k)$ of Galois representations, in which we identify $[V_2]$ with $[V_1]+[V_3]$ only when $V_2\cong V_1\oplus V_3$ as $G_k$-modules.  This is essential since the 
essence of our construction is to distinguish $\one$-valued Galois representations which have the same semisimplification.
Third, we replace $k$ by $k(\zeta_\ell)$ in order to trivialize the cyclotomic character $G_k\to \F_\ell^\times$ (so that $\L$ maps to $1$.)
Up to the $t^\ell$ coefficient, everything works as before, but the expression for $\mu(\zeta_X(t))$ as rational function breaks down at the $t^\ell$ coefficient.  No one $\ell$ value necessary excludes the possibility of rationality but by taking values of $\ell$ tending to infinity, we can prove that $\zeta_X(t)$ cannot be rational.

Assuming the characteristic of $k$ is $0$, we can define $\ntot$ so that for every non-singular projective $k$-variety $X$, we have $\ntot([X]) = [H^\bullet(\bar X,\one)]$ in the Grothendieck ring $\K(G_{k(\zeta_\ell)})$.
It is easy to calculate the semisimplification of $H^\bullet(\overline{\Sym^n X},\one)$ as $G_{k(\zeta_\ell)}$-representation, but as $\Sym^n X$ is in general singular, we do not know
when 
$$\ntot([\Sym^n X]) = [\Sym^n H^\bullet(\bar X,\one)].$$
However, we show that this holds when all the cohomology of $X$ is in even degree and $\ell$ is sufficiently large compared to $n$.
If $\ell$ is large compared to the degrees of the numerator and denominator of $\zeta_X(t)$, then the linear recurrence satisfied by the $\ntot([\Sym^i X])$ ultimately implies
that $\ntot([\Sym^\ell X])$ is non-effective.  This is a result of the breakdown of the correspondence between the (mod $\ell$) representation theory of $\SL_2(\F_\ell)$
and the complex representation theory of $\SL_2(\C)$ which occurs in dimension $\ell$.

Unfortunately, we do not know how to compute the value $\ntot([\Sym^\ell X])$ directly, but using a generalization to arbitrary fields of
G\"ottsche's relation \cite{Go2} in $K_0(\Var_{k})$ between the classes $[X^{[i]}]$ of the Hilbert schemes of $X$ and the classes of the symmetric powers of $X$, we can show that $\ntot([X^{[\ell]}])$ is also non-effective.  This is absurd, since $X^{[\ell]}$ is projective and non-singular.

In \S2, we discuss Grothendieck groups of representations of finite groups, especially $\SL_2(\F_\ell)$ and
$\SL_2(\F_\ell)^2$.  In \S3, we use the method of Bittner \cite{Bi} to construct $\ntot$. 
In \S4, we show that there exists a K3 surface over $\Q$ with the desired Galois-theoretic properties.   
The generalization of G\"ottsche's theorem to arbitrary base field is given in \S5.
In \S6, we discuss some variants of the category of Chow motives which enable us to show that if $\ell$ is large compared to $n$, $\Sym^n X$ behaves like a non-singular variety as far as $\ntot$ is concerned.  The proof of the main theorem is in \S7.

The referee of an earlier version of this paper called our attention to the preprint \cite{Bon}
of Mikhail Bondarko establishing that
Joseph Ayoub's announced proof \cite{Ayoub1,Ayoub2} of the conservativity conjecture implies
the rationality of the motivic zeta-function of any variety in characteristic $0$ with values in the K-group of numerical motives.
We would like to thank the referee for this and many other helpful suggestions.

We would like to gratefully acknowledge helpful conversations with Pierre Deligne, Vladimir Drinfeld, Lothar G\"ottsche, Luc Illusie, Mircea Musta\c{t}\u{a}, and Geordie Williamson. 

\section{Grothendieck rings of representations}

We fix an odd prime $\ell$ and an algebraic closure $\bar\F_\ell$ of the prime field $\F_\ell$, which we regard as a space with the discrete topology.  For any topological group $G$, we denote by $\K(G)$ the Grothendieck ring of the exact category  given by \emph{split} short exact sequences of finite dimensional continuous $\bar\F_\ell[G]$-modules.  

We claim that, as an additive group, $\K(G)$ is the free $\Z$-module on indecomposable continuous $\bar\F_\ell[G]$-modules. 
To see this, recall \cite{Kr} that an additive category is \emph{Krull-Schmidt} if every object is a finite direct sum of indecomposable objects whose endomorphism rings are local.
As every finite-dimensional $G$-module has finite length, the category of such modules is Krull-Schmidt \cite[\S5]{Kr}.

By the Krull-Remak-Schmidt theorem, this implies that the factors appearing in any decomposition into indecomposables, together with their multiplicities, are uniquely determined. 
We say an element of $\K(G)$ is \emph{effective} if it is a non-negative linear combination of  indecomposable classes.

Any continuous homomorphism $G\to H$ induces a restriction homomorphism $\K(H)\to \K(G)$, which maps effective classes to effective classes.  If $G\to H$ is surjective, then 
$\Res_G^H$ is injective because distinct indecomposable representations restrict to distinct  indecomposable representations of $G$.  A class in $\K(H)$ is effective if and only if its restriction to $\K(G)$ is so.

If $G$ is profinite, $\K(G)$ is the direct limit of $\K(G/H)$ as $H$ ranges over open normal subgroups of $G$.
In this section, we consider only finite groups $G$ endowed with the discrete topology, so the continuity condition will play no role.

If $V$ is a $G$-representation, we define $\zeta_V(t)\in \K(G)[[t]]$ as follows:
$$\zeta_V(t) = \sum_{n=0}^\infty [\Sym^n V]\,t^n,$$
where $\Sym^n V$ denotes the space of $\Sigma_n$-coinvariants of the tensor product $V^{\otimes n}$.
Note that the notation makes sense only if $G$ is given.

\begin{prop}
If $H_1$ and $H_2$ are finite groups, external tensor product defines an injective homomorphism $T\colon \K(H_1)\otimes \K(H_2)\to \K(H_1\times H_2)$.
\end{prop}

\begin{proof}
We need to show that if $\rho_1\colon H_1\to \GL(V_1)$ and $\rho_2\colon H_2\to \GL(V_2)$ are indecomposable representations, then $\rho_{12}\colon H_1\times H_2\to \GL(V_1\boxtimes V_2)$ is an indecomposable representation of $H_1\times H_2$ and that moreover, the isomorphism class of the representation $V_1\boxtimes V_2$ determines the isomorphism classes of $V_1$ and $V_2$.  The second claim follows immediately by applying Krull-Remak-Schmidt to the restriction of $V_1\boxtimes V_2$ to $G_1\times \{1\}$ and $\{1\}\times G_2$.

To prove that $\rho_{12}$ is indecomposable, it suffices to prove that the centralizer $Z_{12}$ of the $\bar\F_\ell$-span of $\rho_{12}(H_1\times H_2)$ in $\End(V_1\otimes V_2)$ is a local $\bar\F_\ell$-algebra.
To commute with $\rho_{12}(H_1\times H_2)$ is the same as to commute with $\rho_{12}(H_1\times\{1\})$ and $\rho_{12}(\{1\}\times H_2)$.  If $Z_i$ denotes the centralizer of $\rho_i(H_i)$ in $\End(V_i)$, and $Z'_i$ is any $\one$-linear complement of $Z_i$ in $\End(V_i)$, then the centralizer of $\rho_{12}(H_1\times \{1\})$ is 
$$Z_1\otimes \End(V_2) = Z_1\otimes (Z_2\oplus  Z'_2) = Z_1\otimes Z_2 \oplus Z_1\otimes Z'_2,$$
the centralizer of $\rho_{12}(\{1\}\times H_2)$ is 
$$\End(V_1)\times Z_2 = (Z_1\oplus  Z'_1)\otimes Z_2 = Z_1\otimes Z_2 \oplus Z'_1\otimes Z_2,$$
and the intersection of these two centralizers is $Z_1\otimes Z_2$.  

Each finite-dimensional representation is indecomposable if and only if its endomorphism ring is local \cite[Proposition~5.4]{Kr}.
The tensor product of finite-dimensional local algebras over an algebraically closed field is again local \cite[Theorem 4]{La}, and this proves the proposition.
\end{proof}

We will eventually be interested in the case $G=\SL_2(\F_\ell)^2$, but we start with $H=\SL_2(\F_\ell)$.
We denote by $V_i$  the $i$th symmetric power of the natural $2$-dimensional $\bar\F_\ell$-representation of $H$
and by $\indec$ the representation $V_1\otimes V_{\ell-1}$.

\begin{prop}
\label{facts}
We define
$$F_n \K(H) = 
\begin{cases}
\Span_{\Z}([V_0],\ldots,[V_n])&\text{if $n\le \ell-1$,}\\
\Span_{\Z}([V_0],\ldots,[V_{\ell-1}],[\indec ])&\text{if $n=\ell$,}\\
\K(H)&\text{if $n>\ell$.}
\end{cases}$$
We have the following facts:
\begin{enumerate}
\item The representation $\indec $ is indecomposable.
\item The product on $\K(H)$ is compatible with the filtration $F_i$ in the sense that
$$(F_i \K(H)) \,(F_j \K(H)) \subseteq F_{i+j}\K(H).$$
%
%\item For $n\le \frac{\ell-3}2$, if $\alpha\in \K(H)$ is effective and $\phi(\alpha)\in \phi(F_{n}\K(H))$, then $\alpha\in F_{n}\K(H)$.
\end{enumerate}

\end{prop}

\begin{proof}

The representation $V_1$ is the restriction of the tautological $2$-dimensional representation $\tilde V_1$
of $\SL_2(\bar\F_\ell)$.
Applying \cite[Lemma~3.1,~Proposition~3.3~(iii)]{AJL} with $\lambda = \ell-2$, 
we know that $\tilde V_1\otimes \Sym^{\ell-1}\tilde V_1$
is indecomposable, and by 
\cite[Lemma~4.1~(a)]{AJL}, the restriction $\indec $ of this representation to $\SL_2(\F_\ell)$ is the injective hull of an irreducible
representation of $\SL_2(\F_\ell)$ and therefore indecomposable.
(This fact can also be read off from Table 1 of the same paper.).  This gives claim (1).

By \cite[Lemma 2.5]{AJL}, for $1\le i \le \ell-2$, we have
\begin{equation}
\label{onestep}
V_i\otimes V_1\cong V_{i-1}\oplus V_{i+1}.
\end{equation}

By induction on $j$, this implies that for $i,j\ge 0$ and $i+j \le \ell-1$, we have the Clebsch-Gordan formula
$$V_i\otimes V_j \cong  V_{i+j}\oplus V_{i+j-2}\oplus\cdots\oplus V_{|i-j|}.$$
For $i+j=p$ and $0<i<j$, we claim that 
\begin{equation}
\label{CG}
V_i\otimes V_j= \indec  \oplus \bigoplus_{k=2}^i V_{\ell-2k}.
\end{equation}
The statement is trivial for $i=1$, and for $i\ge 2$,
\begin{align*}
V_{i-2}\otimes V_j \oplus V_i\otimes V_j &\cong (V_{i-2}\oplus V_i)\otimes V_j\\
&\cong V_1\otimes (V_{i-1}\otimes V_{j})\\
&\cong V_1\otimes (V_{\ell-1}\oplus V_{\ell-3}\oplus \cdots\oplus V_{\ell+1-2i}) \\
&\cong \indec  \oplus (V_{\ell-2}\oplus V_{\ell-4})\oplus\cdots\oplus (V_{\ell+2-2i}\oplus V_{\ell-2i}). \\
\end{align*}
As
$$V_{i-2}\otimes V_j = V_{\ell-2}\oplus V_{\ell-4} \oplus \cdots \oplus V_{\ell+2-2i},$$
Krull-Schmidt implies our claim, which in turn implies (2).
%For (4), it suffices to prove that any indecomposable representation $V$ whose 
%Jordan-H\"older constituents all lie in $\{V_i\mid 0\le i\le \frac{\ell-3}2\}$
%must be indecomposable.  By the fact (see the comment following \cite[Corollary 4.3]{AJL}) 
%that  $\Ext^1(V_i,V_j)=0$ for $i+j<\ell-2$, this claim holds if $V$ has exactly two factors;
%the general case follows by induction on the number of factors.
\end{proof}

Let
$$\Lambda_{V_1}(t) := 1 - [V_1] t + t^2 \in \K(H)[t].$$
The analogy between the (mod $\ell$) representation theory of $H$ and the (complex) representation theory of
$\SL_2(\C)$ might suggest the possibility that $\zeta_{V_1}(t) = \Lambda_{V_1}(t)^{-1}$, but this turns out not to be true.
Instead (\ref{onestep}) implies
\begin{equation}
\label{toy}
\Lambda_{V_1}(t)\zeta_{V_1}(t) \equiv  1 + ([V_{\ell-2}] + [V_\ell] - [\indec])t^\ell\pmod{t^{\ell+1}}.
\end{equation}
Note that since $\indec$ is indecomposable, the $t^\ell$ coefficient of $\Lambda(t)\zeta_{V_1}(t)$ is non-zero.
This phenomenon, as it arises in the case of the representation $V_1\boxtimes V_1$ of $\SL_2(\F_\ell)\times \SL_2(\F_\ell)$
is the key to our proof of irrationality.

Henceforth $G = \SL_2(\F_\ell)\times \SL_2(\F_\ell)$.
For non-negative integers $n$, we define 
$$F_n \K(G) := T(F_n \K(H)\otimes F_n \K(H)).$$
In particular, for $0\le n \le \ell-1$, 
$$F_n\K(G) = \Span_{\Z}\{[V_i\boxtimes V_j]\mid 0\le i,j \le n\}.$$

\begin{prop}
For $0\le n\le \ell-1$, we have
$$\Sym^n (V_1\boxtimes V_1) \cong \sum_{i=0}^{\lfloor n/2\rfloor} V_{n-2i}\boxtimes V_{n-2i}.$$
\end{prop}

\begin{proof}
First of all, the symmetric power is a quotient of 
$$(V_1\boxtimes V_1)^{\otimes n} = V_1^{\otimes n}\boxtimes V_1^{\otimes n},$$
which by (\ref{onestep}) and induction on $n$ is a direct sum of expressions of the form $V_i\boxtimes V_j$ with $i,j\le n$.  Thus
$\Sym^n (V_1\boxtimes V_1)$ is itself a direct sum of such expressions.  Writing
$$\Sym^n (V_1\boxtimes V_1) = \bigoplus_{0\le i,j\le n} (V_i\boxtimes V_j)^{a_{i,j}},$$
it remains to prove that $a_{i,j}$ is $0$ except when $i=j\in \{n,n-2,n-4,\ldots\}$, in which case it is $1$.

Restricting to $H\times \{1\}$, we obtain the isomorphism of $H$-modules
\begin{equation}
\label{restriction}
\bigoplus_{0\le i,j\le n} V_i^{a_{i,j}(j+1)} \cong \Sym^n(V_1\oplus V_1) \cong \bigoplus_{a+b=n} V_a\otimes V_b \cong \bigoplus_{k=0}^{\lfloor n/2\rfloor} V_{n-2k}^{n-2k+1},
\end{equation}
the last isomorphism following from (\ref{CG}).
Thus, $a_{i,j}(j+1)\le i+1$ for all $i,j\le n$.  By symmetry, also $a_{i,j}(i+1)\le j+1$.  Thus, $a_{i,j}\le 1$ with equality only if $i=j$.   Comparing with (\ref{restriction}), we see that $a_{i,i}=1$ exactly for $i\in \{n,n-2,n-4,\ldots\}$.
\end{proof}

\begin{prop}
\label{recur}
Define 
\begin{equation}
\label{Lambda}
\Lambda_{V_1\boxtimes V_1}(t) := 1-[V_1\boxtimes V_1]\,t + \bigl([V_2\boxtimes V_0]+[V_0\boxtimes V_2]\bigr)\,t^2  - [V_1\boxtimes V_1]\,t^3 + t^4.
\end{equation}
Then
$$\Lambda_{V_1\boxtimes V_1}(t)\zeta_{V_1\boxtimes V_1}(t)\equiv 1 \pmod{t^\ell}.$$
\end{prop}

\begin{proof}
This is the special case $V := V_1\boxtimes V_1$ of the general congruence formula
$$\biggl( \sum_{i=0}^{\dim V} (-1)^i [\wedge^iV]t^i\biggr)\zeta_V(t)
 \equiv 1\pmod{t^\ell}.$$
Equivalently, we claim that for $1\le k < l$, we have
\begin{equation}
\label{partial-lambda}
\sum_{i+j=k} (-1)^i [\wedge^i V\otimes \Sym^j V] = 0.
\end{equation}

For every object $W$ of a $\lambda$-ring, we have the identity
$$\sum_{i+j = k}   \lambda^i(W) \lambda^j(-W) = 0.$$
If $W$ is a finite-dimensional complex vector space regarded as an object of the representation ring
of $\GL(W)$, it is easy to see by the splitting principle that $(-1)^j\lambda^j (-W)$ corresponds to $\Sym^j W$.
If $\epsilon_{i,j}$ in the group ring $\Z[1/k!][S_k]$ denotes the projector which maps $W^{\otimes k} = W^{\otimes i+j}$ onto 
$\wedge^i W\otimes \Sym^j W$, this implies
$$\sum_{i+j=k} (-1)^i\epsilon_{i,j} = (-1)^{i+j} \sum_{i+j=k} (-1)^j\epsilon_{i,j} = 0.$$
As $k!$ is invertible in $\ell$, this reduces to the same identity over $\F_l$, which implies
the identity (\ref{partial-lambda}) for group representations in characteristic $l$.
\end{proof}

We now come to the key lemma.
Let $R$ be a ring containing $\K(G)$.
Let $A(t),B(t)\in R[t]$ denote polynomials with $A(0)=B(0)=1$, and let $\one^k$ denote the trivial representation of 
$G$ of dimension $k$.  

\begin{lem}
\label{key}
If $A(t)$ and $B(t)$ are as above, $a_i\in R$ for all $i\ge 0$,
\begin{equation}
\label{fracAB}
A(t) = B(t)(a_0+a_1t + a_2 t^2 + \cdots),
\end{equation}
and
\begin{equation}
\label{symform}
a_i = [\Sym^i((V_1\boxtimes V_1)\oplus \one^k)] \in \K(G)
\end{equation}
for $i\le \deg A+\deg B+k+4 < \ell$, then $a_i\in \K(G)$ for all $i\ge 0$, and
$$a_\ell - [\indec\boxtimes \indec] + [\indec\boxtimes V_{\ell-2}] + [V_{\ell-2}\boxtimes \indec]\in F_{\ell-1}\K(G).$$
\end{lem}

\begin{proof}
For any $G$-representation $V$,
$$\Sym^n(V\oplus \one) \cong \bigoplus_{i=0}^n \Sym^i V.$$
Thus, 
$$(1-t)\sum_{i=0}^n [\Sym^i(V\oplus \one)]\,t^i\equiv \sum_{i=0}^n [\Sym^i V]\,t^i\pmod{t^{n-1}}.$$
Iterating,
\begin{align*}
(1-t)^k\sum_{i=0}^\infty a_it^i &\equiv (1-t)^k\sum_{i=0}^n [\Sym^i((V_1\boxtimes V_1)\oplus \one^k)]\,t^i\\
&\equiv \sum_{i=0}^n [\Sym^i(V_1\boxtimes V_1)]\,t^i\pmod{t^{n+1}}.
\end{align*}
Thus, replacing $B(t)$ with $B(t)(1-t)^k$, we may assume $k=0$,
which means $a_i = [\Sym^i(V_1\boxtimes V_1)]$ for $0\le i\le \deg A + \deg B+4$.

Defining $\Lambda_{V_1\boxtimes V_1}(t)$ as in (\ref{Lambda}) and multiplying (\ref{fracAB}) by $B(t)\Lambda_{V_1\boxtimes V_1}(t)$, we get
\begin{align*}
A(t)\Lambda_{V_1\boxtimes V_1}(t)\equiv B(t)\Lambda_{V_1\boxtimes V_1}(t)\sum_{i=0}^\infty a_i t^i &\equiv B(t)\Lambda_{V_1\boxtimes V_1}(t) \sum_{i=0}^\infty [\Sym^i(V_1\boxtimes V_1)]\, t^i \\
&\equiv B(t) \pmod{t^{\deg A + \deg B +5}}.
\end{align*}
As $\deg(A(t)\Lambda_{V_1\boxtimes V_1}(t)-B(t)) < \deg A +\deg B+5$, we have $A(t)\Lambda_{V_1\boxtimes V_1}(t) = B(t)$, so
\begin{equation}
\label{inv}
A(t) (1 - \Lambda_{V_1\boxtimes V_1}(t)\sum_{i=0}^\infty a_i t^i) = 0.
\end{equation}
As $A(t)$ is invertible in $R[[t]]$, this implies 
$$\Lambda_{V_1\boxtimes V_1}(t)\sum_{i=0}^\infty a_i t^i = 1.$$
This gives a linear recurrence for the $a_i$ with coefficients in $\K(G)$, and it follows that $a_i\in \K(G)$ for all $i\ge 0$.

By Proposition~\ref{recur},
$$\zeta_{V_1\boxtimes V_1}(t) \equiv \sum_{i=0}^\infty a_i t^i\pmod{t^\ell}.$$
This implies (\ref{symform}) for  all $n\le \ell-1$.

Finally, matching $t^\ell$ coefficients in (\ref{inv}), we get
$$a_\ell = a_1 a_{\ell-1} - \bigl([V_2\boxtimes V_0]+[V_0\boxtimes V_2]\bigr) a_{\ell-2} + a_1 a_{\ell-3} - a_{\ell-4}.$$
Modulo classes in $F_{\ell-1}\K(G)$, the right hand side reads
$$[V_1\boxtimes V_1]\,[V_{\ell-1}\boxtimes V_{\ell-1}] - \bigl([V_2\boxtimes V_0]+[V_0\boxtimes V_2]\bigr)[V_{\ell-2}\boxtimes V_{\ell-2}],$$
which,  by (\ref{CG}), further reduces modulo $\Span_{\Z}\{[V_i\boxtimes V_j]\mid 0\le i,j\le \ell-1\}$ to 
$$[\indec\boxtimes \indec] - [\indec\boxtimes V_{\ell-2}] - [V_{\ell-2}\boxtimes \indec].$$

\end{proof}

\begin{lem}
\label{keyb}
Suppose the hypotheses of Lemma~\ref{key} are satisfied.   If elements $b_i\in R$ satisfy
\begin{equation}
\label{prod}
\sum_{i=0}^\infty b_i t^i = \prod_{r=1}^\infty \sum_{j=0}^\infty a_j t^{jr}
\end{equation}
then $b_i\in \K(G)$ for all $i\ge 0$, and
$$b_\ell\equiv [\indec\boxtimes \indec] - [\indec\boxtimes V_{\ell-2}] - [V_{\ell-2}\boxtimes \indec]\pmod{F_{\ell-1}\K(G)}.$$
\end{lem}

\begin{proof}
Since $a_i\in \K(G)$ for all $i$ and $a_i\in F_i\K(G)$ for $0\le i\le \ell-1$, (\ref{prod}) implies by induction on $i$ that $b_i\in \K(G)$ for all $i\ge 0$ and $b_i\in F_i\K(G)$ for $0\le i\le \ell-1$.
As $b_\ell - a_\ell$ is a linear combination of products $a_{i_1}\cdots a_{i_k}$ with $i_1+\cdots+i_n \le \ell-1$,
the claim for $b_\ell$ follows.

\end{proof}

We note for future reference that the relationship (\ref{prod}) between the $a_i$ and the $b_i$  is significant because 
it is the relationship between the cohomology of the $i$th symmetric power of $X$ and the $i$th Hilbert scheme of $X$.

\section{A family of motivic measures}
In this section, we construct the motivic measures needed for the proof of our main theorem.

Let $k$ be a subfield of $\C$.  Let $\bar k$ denote the algebraic closure of $k$ in $\C$,  and set $G_k := \Gal(\bar k/k)$.
We define $\bar X := X\times_{\Spec k} \Spec \bar k$ for any variety $X/k$.  We regard the \'etale cohomology groups
$H^i(\bar X,\one)$ and $H_c^i(\bar X,\one)$ as $G_k$-representations.  They are obtained by extension of scalars from
the $G_k$-representations  $H^i(\bar X,\F_\ell)$ and $H_c^i(\bar X,\F_\ell)$ respectively.

Our construction depends on the Bittner construction \cite{Bi}.  In order to carry it out, we need the following theorem:

\begin{thm}
\label{blowup}
Let $X$ be a non-singular projective variety over $k$ and $Y\subset  X$ a
nonsingular closed subvariety of codimension $r$, $X'$ the blow up of $X$ along $Y$, and $Y'$
the inverse image of $Y$ in $X'$. Then for any $q$ there is a natural
direct sum decomposition of $G_k$-modules
$$H^q(\bar X',\one)=H^q(\bar X,\one)\oplus  \bigoplus _{j=1}^{r-1}H^{q-2j}(\bar Y,\one (-j))$$
\end{thm}

The analogue Theorem~\ref{blowup}  is proved for $\Z _\ell$-coefficients (instead of $\one$) in
\cite[XVIII 2.2.2]{SGA7}. We essentially reproduce the  argument (filling in some details) to prove it in our setting.
We make use of the 
following proposition \cite[VI, 10.1]{Mi}:

\begin{prop} \label{milne} Let $\bar Y$ be a smooth projective variety over $\bar{k}$ and let
$\cE$ be a vector bundle of rank $r$ over $\bar Y$. Let $\bbP (\cE)\to \bar Y$ be the corresponding
projective bundle. Then for each $q$ there is a natural isomorphism of $G_k$-modules
$$H^q(\bbP (\cE),\one)=\bigoplus _{j\geq 0}H^{q-2j}(\bar Y,\one(-j)),$$
where the summand for $j=0$ is the image of the map $p^*\colon H^q(\bar Y,\one)\to H^q(\bbP(\cE),\one)$.
\end{prop}

\begin{proof}[Proof of Theorem~\ref{blowup}] Consider the commutative diagram
\begin{equation*}
\minCDarrowwidth20pt\begin{CD}
\bar Y' @>>> \bar X'\\
@VgVV @VVfV\\
\bar Y @>>>\bar  X
\end{CD}
\end{equation*}Note that the map $g: \bar Y'\to \bar Y$ is the projective bundle corresponding to the normal vector bundle $N_{\bar Y/\bar X}$ 
of rank $r$ on $\bar Y$.  Writing $\bar U := \bar X\setminus \bar Y \cong \bar X'\setminus \bar Y' =: \bar U'$,
we have the induced morphism of long exact sequences of cohomology with compact supports
\begin{equation}
\label{diagr}
\minCDarrowwidth8pt\begin{CD}
H^i_c(\bar U',\one) @>>> \  H^i(\bar X',\one) @>>>  H^i(\bar Y',\one) @>>> H^{i+1}_c(\bar U',\one) @>>> H^{i+1}(\bar X',\one) \\
@| @AAf^*A @AAg^*A @| @AAf^*A \\
H^i_c(\bar U,\one) @>>> \  H^i({\bar X},\one) @>>>  H^i({\bar Y},\one) @>>> H^{i+1}_c(\bar U,\one) @>>> H^{i+1}({\bar X},\one) \\
\end{CD}
\end{equation}
(where we use the fact that $H^i_c(-)$ coincides with $H^i(-)$ for the projective varieties $\bar X,\bar X',\bar Y, \bar Y'$).

We know that the map $g^*$ is injective by Proposition~\ref{milne}. We claim that the arrows $f^*$ are also injective. 

\begin{lemma} \label{lemma} For each $i$ the map $f^*:H^i (\bar X,\one)\to H^i(\bar X',\one)$ 
is injective. Moreover it has a natural left inverse.
\end{lemma}

\begin{proof} We may assume that $\bar X$ and hence also $\bar X'$ is connected, 
so that $H^0(\bar X,\one)=H^0(\bar X',\one)=\one$. We denote by $1_{\bar X}\in H^0(\bar X,\one)$ and $1_{\bar X'}\in H^0(\bar X',\one)$ the corresponding generators. Clearly $f^*(1_{\bar X})=1_{\bar X'}$. 

Let $n=\dim \bar X=\dim \bar X'$. 
Recall that Poincar\'e duality  \cite[VI, Theorem 11.1]{Mi} gives a canonical nondegenerate Galois-equivariant pairing 
$$H^r(\bar X,\one)\times H^{2n-r}(\bar X,\one(n))\stackrel{\cup}{\longrightarrow} H^{2n}(\bar X,\one (n))\stackrel{\eta_{\bar X}}{\longrightarrow} \one$$
where the map on the left is cup-product and
$\eta_{\bar X}:H^{2n}(\bar X,\one(n))\to \one$ is the trace map isomorphism. It has the property that for every closed point $P$ we have
$\eta_{\bar X}(\cycl_X(P))=1\in \one$, where $\cycl_{\bar X}(P)\in H^{2n}(\bar X,\one(n))$ is the image under the Gysin map $H^0(P,\one)\to H^{2n}(\bar X,\one(n))$ of $1_P$ \cite[VI, p.269]{Mi}. The same applies to $\bar X'$.  

Choose a closed point $P\in \bar X\setminus \bar Y$. Then $f^{-1}(P)=Q$ is a single point in $\bar X'$. Hence it follows from \cite[VI, Proposition 9.2]{Mi} that the map $f^*:H^{2n}(\bar X,\one(n))\to H^{2n}(\bar X',\one(n))$ takes $\cycl_{\bar X}(P)$ to $\cycl_{\bar X'}(Q)$. 
It follows that the diagram 
\begin{equation*}\begin{CD}
H^{2n}(\bar X',\one(n)) @>\eta_{\bar X'}>>  \one\\
@Af^*AA @| \\
H^{2n}(\bar X,\one(n)) @>\eta_{\bar X}>>  \one
\end{CD}
\end{equation*}
commutes. 

Poincar\'e dualities for $H^\bullet (\bar X)$ and for $H^\bullet (\bar X')$ induce the pushforward map 
$$f_*:H^i(\bar X',\one)\to H^i(\bar X,\one)$$
such that $\f_*(x' \cup f^*(x))=f_*(x')\cup x$ for $x\in H^\bullet (\bar X)$ and 
$x'\in H^\bullet (\bar X')$ \cite[VI, Remark 11.6]{Mi}. We have $f_*(1_{\bar X'})=d\cdot 1_{\bar X}$ for some $d\in \one$. 
We claim that in fact $d=1$. Indeed, in the above notation we have
$$\begin{array}{rcl}
1 & = & \eta _{\bar X'}(\cycl_{\bar X'}(Q))\\
  & = & \eta _{\bar X'}(1_{\bar X'}\cup \cycl_{\bar X'}(Q))\\
  & = & \eta _{\bar X'}(1_{\bar X'}\cup f^*(\cycl_{\bar X}(P)))\\
  & = & \eta _{\bar X}(f_*(1_{\bar X'})\cup \cycl_{\bar X}(P))\\
  & = & d\cdot \eta_{\bar X}(1_{\bar X}\cup \cycl_{\bar X}(P))\\
  & = & d\cdot \eta_{\bar X}(\cycl_{\bar X}(P))\\
  & = & d
\end{array}
$$

It follows that $f_*f^*(x)=x$ for every $x\in H^\bullet(X,\one)$. Indeed,
$$f_*f^*(x)=f_*(1_{X'}\cup f^*(x))=f_*(1_{X'})\cup x=1_X\cup x=x$$
This proves the lemma.
\end{proof}   

It follows from the injectivity of $f^*$ that the diagram (\ref{diagr}) induces an isomorphism 
$$\coker f^* \isom\coker g^*.$$
The fact that $f^*$ has a canonical left-inverse allows us to identify $H^i(\bar X',\one)$ with $H^i(\bar X,\one)\oplus \coker f^*$
as $G_k$-modules.
Thus 
$$H^q(\bar X',\one)\isom H^q(\bar X,\one)\oplus \coker g^*\isom
H^q(\bar X,\one)\oplus\bigoplus_{j=1}^{r-1} H^{q-2j}(\bar Y,\one(-j))$$
by Proposition~\ref{milne}.
\end{proof}

\begin{thm}
For each prime $\ell$ and every field $k$ of characteristic $0$,  there exists a unique motivic measure $\mtot\colon K_0(\Var_k)\to \K(G_k)$ 
satisfying
$$\mtot([X]) = \sum_{i=0}^{2\dim X} [H^i(\bar X,\one)].$$
for all projective non-singular varieties $X$.
\end{thm}

\begin{proof}
By Bittner's theorem \cite{Bi}, it suffices to prove  that $\mtot([X\times Y]) = \mtot([X])\mtot([Y])$ whenever $X$ and $Y$ are non-singular projective varieties and that whenever  $X$ is a non-singular projective variety, $Y$ a non-singular closed subvariety, $X'$ the blow up of $X$ along $Y$ and $Y'$ the inverse image of $Y$ in $X'$ then
$$\mtot([X'])-\mtot([X]) = \mtot([Y'])-\mtot([Y]).$$
The first property follows immediately from the K\"unneth formula \cite[VI, 8.13]{Mi}.
The second follows from Theorem~\ref{blowup} and Proposition~\ref{milne}.

\end{proof}

\begin{defn}
We define the motivic measure $\ntot\colon K_0(\Var_k)\to \K(G_{k(\zeta_\ell)})$ to be the composition of $\mtot$ with the restriction map
$\K(G_{k})\to \K(G_{k(\zeta_\ell)})$.  
\end{defn}
In the application to the main theorem, we will always take $k=\Q$.
\section{Galois representations}

\begin{prop}
\label{E1E2}
There exist elliptic curves $E_1$ and $E_2$ over $\Q$ such that for all sufficiently large primes $\ell$,
there exist linearly disjoint Galois extensions $K_1$ and $K_2$ of $\Q(\zeta_\ell)$ such that the (mod $\ell$)
Galois representations of $G_{\Q(\zeta_\ell)}$ acting on $H^1(\bar E_i,\F_\ell)$ have kernels $G_{K_i}$
and images isomorphic to $\SL_2(\F_\ell)$.
\end{prop}

\begin{proof}
Fix primes $q,r\ge 5$.
Let $E_1$ and $E_2$ be any elliptic curves over $\Q$ with multiplicative reduction at $q$
and such that $E_1$ and $E_2$ have respectively good ordinary reduction and good supersingular reduction at $r$.
(For instance, if $q=11$ and $r=5$, the curves given in Cremona notation by $E_1 := \mathtt{33a1}$ and $E_2 := \mathtt{11a1}$ satisfy these conditions.)
Let $\rho^\ell_i$ denote the homomorphism from the absolute Galois group $G_{\Q}$
to $\GL(H^1(\bar E_i,\F_\ell))\cong \GL_2(\F_\ell)$.

Neither $E_1$ nor $E_2$ can have complex multiplication, since every CM curve has integral $j$-invariant \cite[II~Theorem~6.1]{Si},
while an elliptic curve with multiplicative reduction at $q$ cannot have $q$-adically integral $j$-invariant \cite[Table~4.1]{Si}.
By Serre's theorem \cite{Se}, for $\ell$ sufficiently large, the $\rho^\ell_i$ is surjective, 
so the image of  $G_{\Q(\zeta_\ell)}$ in  $\GL(H^1(\bar E_i,\F_\ell))$ is $\SL_2(\F_\ell)$.
We assume this holds and that $\ell\ge 5$.
Let $\bar\rho^\ell_i\colon G_{\Q}\to \PGL_2(\F_\ell)$ denote the composition of $\rho^\ell_i$ with the quotient map
$\GL_2(\F_\ell)\to \PGL_2(\F_\ell)$.  

Suppose that $\rho^\ell_1|_{G_{\Q(\zeta_\ell)}} =  \rho^\ell_2|_{G_{\Q(\zeta_\ell)}}$.  
As the common image of the two representations has trivial centralizer in $\PGL_2(\F_\ell)$, it follows that
$\bar\rho^\ell_1=\bar\rho^\ell_2$.  Thus, $\rho^\ell_1 = \rho^\ell_2\otimes\chi$ for some character  $\chi$ of $\Gal(\Q(\zeta_\ell)/\Q)$.  Taking determinant of both sides, $\chi^2=1$.

The representations $\rho^\ell_i$ are both unramified at $r$, so $\Tr(\rho^\ell_i(\mathrm{Frob}_r))$ is well defined, and the two traces are related by a factor of $\chi(\mathrm{Frob}_r)=\pm 1$.  This is impossible since the trace of $\mathrm{Frob}_r$ is zero for $E_2$ but not for $E_1$.

Now $\rho^\ell_1$ and $\rho^\ell_2$ together give an injective homomorphism $\rho^\ell_{12}$
$$\Gal(K_1K_2/\Q(\zeta_\ell))\to \SL_2(\F_\ell)\times \SL_2(\F_\ell)$$
whose image projects onto $\SL_2(\F_\ell)$ on both factors.  As the only normal subgroups of $\SL_2(\F_\ell)$
are the group itself, $\{\pm 1\}$ and $\{1\}$, applying Goursat's lemma to the image of $\rho_{12}^l$, 
either this image is all of $\SL_2(\F_\ell)\times \SL_2(\F_\ell)$, in which case $\rho_{12}^\ell$ is an isomorphism, or 
$\bar\rho_1$ and $\bar\rho_2$ coincide on $\Gal(K_1 K_2/\Q(\zeta_\ell))$.  We have seen that the latter is impossible, so the proposition follows.
\end{proof}

We remark that assuming the Frey-Mazur conjecture is true, Proposition~\ref{E1E2} is true for any two non-CM elliptic curves which are not isogenous over $\bar\Q$.

If $\ell$ is a prime, $\Sigma$ is a group of order prime to $\ell$, and $V$ a finite dimensional $\Sigma$-representation over $\one$, then the map
$$v\mapsto |\Sigma|^{-1}\sum_{\sigma\in\Sigma} \sigma v$$
induces a natural isomorphism $V_\Sigma\to V^\Sigma$, from coinvariants to invariants.  In what follows, we do not distinguish between
these spaces; in particular, we identify the symmetric $n$th power with the symmetric tensors when $n<\ell$.
\begin{prop}
\label{quotient}
Let $X$ be a variety over a field $k\subseteq \C$, $\ell$ a prime, and $\Sigma$ a finite group whose order is prime to $\ell$.
For any $\Sigma$-action on $X$ defined over $k$ and for every $q\ge 0$, there is a canonical isomorphism of $G_k$-modules
$$H^q(\overline{X/\Sigma},\one)\isom H^q(\bar X,\one)^\Sigma.$$
\end{prop}

\begin{proof}
As the morphism $X\to X/\Sigma$ is defined over $k$, the natural map $H^q(\overline{X/\Sigma},\one)\to H^q(\bar X,\one)$
respects $G_k$-actions, and its image of course lies in the space of $\Sigma$-invariants.  It remains to see that this is an isomorphism.
As $k\subseteq \C$, we can use the comparison theorem with the ordinary cohomology of $X(\C)$.  The  statement for cohomology of CW complexes is due to Grothendieck \cite[5.2.3]{Tohoku}.
\end{proof}

\begin{cor}
\label{Symm}
Let $X$ be a variety over a field $k\subseteq \C$, $\ell$ a prime, and $n<\ell$ is a non-negative integer, then
$$H^\bullet(\overline{\Sym^n X},\one) \isom \Sym^n H^\bullet(\bar X,\one),$$
where $\Sym^n$ is taken in the sense of $\Z/2\Z$-graded $G_k$-representations (i.e., if $V = V^0\oplus V^1$, $\Sym^n V$ means 
$\bigoplus_{i+j=n} \Sym^i V^0\otimes \wedge^j V^1$.)
\end{cor}

\begin{proof}
By the K\"unneth formula, 
$$H^\bullet(\overline{X^n},\one) \isom H^\bullet(\bar X,\one)^{\otimes n}.$$
The corollary follows by applying Proposition~\ref{quotient} to $\Sigma := \Sigma_n$.
\end{proof}

Note that  if $H^\bullet(\bar X,\one)$ is zero in odd degrees, then the action of $\Sigma_n$ is the usual permutation action on tensor factors, and the
symmetric $n$th power can therefore be taken in the usual sense of $G_k$-representations. There is no distinction between the alternating
sum of cohomology and the total cohomology so we can 
work with Galois representations rather than virtual representations.

\begin{thm}
\label{K3}
Let $E_1$, $E_2$, $K_1$, $K_2$ be as in Proposition~\ref{E1E2}.  Let $X'$ denote the K3 surface obtained by blowing up the nodes of the Kummer surface 
$$X := (E_1\times E_2)/\langle \iota\rangle,$$  
where $\iota$ is multiplication by $-1$.
For $\ell$ sufficiently large, there is an isomorphism
\begin{equation}
\label{key-group}
\Gal(K_1 K_2/\Q(\zeta_\ell))\to \SL_2(\F_\ell)\times \SL_2(\F_\ell),
\end{equation}
and with respect to this isomorphism,
$$\ntot([X']) = \Res_{G_{\Q(\zeta_\ell)}}^{\SL_2(\F_\ell)^2}[\one^{20}\oplus V_1\boxtimes V_1].$$
\end{thm}

\begin{proof}
The action of $\iota$
on the $\ell$-torsion of $E_1\times E_2$ and therefore on 
$$H^1(\overline{E_1\times E_2},\one)\isom \Hom(\bar E_1[\ell]\times \bar E_2[\ell],\one)$$
is by multiplication by $-1$; as the cohomology of every abelian variety is generated by $H^1$, $\iota$ acts on 
$H^q(\overline{E_1\times E_2},\one)$ by multiplication by $(-1)^q$.  Assuming $\ell>2$, It follows from Proposition~\ref{quotient}
that $H^q(\bar X,\one)$ is $0$ for $q$ odd and is 
$$H^2(\overline{E_1\times E_2},\one) \cong \one(1)\oplus H^1(\bar E_1,\one)\otimes H^1(\bar E_2,\one)\oplus \one(1)$$
for $q=2$.  For $q=0$ and $q=4$, we get $\one$ and $\one(2)$ respectively.

Let $Y$ denote the set of $16$ double points on $X$ and $Y'$ the inverse image of $Y$ in $X'$, consisting of $16$ copies of $\bbP^1$.  Let $U := X\setminus Y \cong X'\setminus Y'$.  The excision sequence for $U\subset X$ gives
$H^i_c(\bar U,\one)\isom H^i(\bar X,\one)$ for $i\ge 2$, and if $\ell$ is sufficiently large, the excision sequence for $U\subset X'$ gives a short exact sequence of $G_{\Q(\zeta_\ell)}$-modules
$$0\to H^2(\bar X,\one)\to H^2(\bar X',\one)\to \one(1)^{16}\to 0$$
and therefore
$$0\to H^1(\bar E_1,\one)\boxtimes H^1(\bar E_2,\one)\to H^2(\bar X',\one)\to \one(1)^{18}\to 0.$$
Regarding $H^2(\bar X',\one)$ as a representation of $G_{\Q(\zeta_\ell)}$, it factors through the Galois group $\Gal(K_1K_2/\Q(\zeta_\ell))$
which is isomorphic to $\SL_2(\F_\ell)^2$.  As $\SL_2(\F_\ell)^2$-representation it is an extension of an $18$-dimensional trivial representation by $V_1\boxtimes V_1$.  If $\ell$ is sufficiently large, this extension is trivial, since all indecomposable $\one$-representations of $\SL_2(\F_\ell)$ which are not irreducible have dimension at least $\ell-2$ \cite[Corollary~4.3]{AJL}.
As $H^0(\bar X',\one)$ and $H^4(\bar X',\one)$ are trivial one-dimensional representations of $G_{\Q(\zeta_\ell)}$ and
$H^1(\bar X',\one) = H^3(\bar X',\one) = 0$, the theorem follows.

\end{proof}

\section{Hilbert schemes of surfaces}

\newcommand{\Imm}{\mathrm{Imm}}
\newcommand{\fm}{\mathfrak{m}}
\def\open#1{{\breve X}^{#1}}
\def\X#1{{X^{\{#1\}}}}
\def\XX#1#2{{X_{#2}^{\{#1\}}}}

This section is devoted to a proof of an identity relating the classes of the Hilbert schemes of a non-singular surface to those of symmetric powers of the surface.
\begin{thm}
\label{Gottsche-Th}
If $X$ is a non-singular surface over a field $k$, we have an identity of power series in 
$K_0(\Var_k)$ as follows:
\begin{equation}
\label{Gottsche}
\sum_{n=0}^\infty \bigl[X^{[n]}\bigr]t^n = \prod_{i=1}^\infty \zeta_X(\L^{i-1}t^i).
\end{equation}
\end{thm}
This theorem is due to G\"ottsche \cite{Go2} in the case that $k$ is algebraically closed and of characteristic zero.  
His proof goes through in essentially the same way for any field.
The only point requiring elaboration is the key identity (due to Ellingsrud and Str\o mme \cite[Theorem~1.1(iv)]{ES} in the case $k=\C$) for the class in  $K_0(\Var_{k})$ of $[R_n]$.
Here $R_n$ denotes the (reduced) Hilbert scheme of codimension $n$ ideals of $k[[x,y]]$, or, equivalently, in $k[[x,y]]/(x,y)^n$, or, yet again, if $R$ is any $k$-algebra of Krull dimension $2$ and
$\fm$ any maximal ideal of $R$ such that $\dim_k R/\fm=1$ and $R$ is regular at $\fm$, the Hilbert scheme of $\fm$-primary codimension $n$ ideals of $R$.
It is convenient to take $R=k[x,y]$ and $\fm=(x,y)$.

The following identity holds for general $k$.

\begin{prop}
\label{ES}
Let $k$ be any field, and let $n$ be a positive integer.  Then
$$[R_n] = \sum_{\beta\in P(n)} [\A^{n-|\beta|}],$$
where $P(n)$ denotes the set of partitions $\beta$ of $n$, and $|\beta|$ is the number of parts of a partition.
\end{prop}

We prove the proposition by giving an explicit ``cell decomposition'' of $R_n$ and explicit parametrizations 
of the cells.  Toward this end, we introduce the following notation. 
Let $\beta$ and $\lambda$ be mutually dual partitions (i.e., partitions whose Ferrers diagrams are transpose to one another) with
\begin{align*}
r&=\beta_1\ge \beta_2\ge\cdots\ge \beta_s > \beta_{s+1}=0 \\
s&=\lambda_1\ge \lambda_2\ge \cdots\ge \lambda_r > \lambda_{r+1}=0.  
\end{align*}
Thus,
$$\beta_{\lambda_i+1}  < i \le \beta_{\lambda_i}$$
for $1\le i\le r$, and 
$$\lambda_{\beta_j+1} < j \le \lambda_{\beta_j}$$
for $1\le j\le s$.
For $\beta$ (and therefore $\lambda$) fixed, we define the polynomial ring $\cA_\beta := \Z[t_{ij}]$ where $1\le i< r$ and $1\le j \le  \lambda_{i+1}$
and recursively define (working from bottom right to top left as in the example depicted below)
the finite sequences of polynomials $Q_1, Q_2,\ldots,Q_{r+1}=1$ and $P_1=1,P_2,\ldots P_s$ in $\cA_\beta[x,y]$ as follows: for $1\le i \le r$,
$$Q_i := y^{\lambda_{i}-\lambda_{i+1}} Q_{i+1} + \sum_{j=1}^{\lambda_{i+1}} t_{ij} x^{\beta_j-i} P_j,$$
and for $1\le j\le s$,
$$P_j := y^{j-\lambda_{\beta_j+1}-1}Q_{\beta_j+1}.$$
\medskip
\begin{center}
\begin{tikzpicture}[scale=2, line width=1pt]
  \draw (0,0) grid (5,1);
  \draw (0,0) grid (3,2);
  \draw (0,0) grid (2,3);
  \draw (0,0) grid (2,4);
  \draw (0,0) grid (1,5);
  \node[right] at (5.1,.45) {$Q_6=1$};
  \node[right] at (4.1,1.2) {$Q_5=yQ_6$};
  \node[right] at (3.1,1.45) {$Q_4=Q_5+t_{41}xP_1$};
  \node[right] at (2.1,2.45) {$Q_3=yQ_4+t_{31}x^2P_1$};
  \node[right] at (1.1,4.45) {$Q_2=y^2Q_3+t_{21}x^3P_1+t_{22}xP_2$};
  \node[right] at (.1,5.45) {$Q_1=yQ_2+t_{11}x^4P_1+t_{12}x^2P_2+t_{13}xP_3+t_{14}xP_4$};
  \node[right] at (4.1,.45) {$P_1=Q_6$};
  \node[right] at (2.1,1.45) {$P_2=Q_4$};
  \node[right] at (1.1,2.45) {$P_3=Q_3$};
  \node[right] at (1.1,3.45) {$P_4=yQ_3$};
  \node[right] at (.1,4.45) {$P_5=Q_2$};

\end{tikzpicture}
\end{center}
\medskip
As $\beta_j\ge i+1$ when $j\le \lambda_{i+1}$, by descending induction, for  $1\le i \le r+1$
$$Q_i\in y^{\lambda_i} + (x),$$
and by (standard) induction it follows that 
$$P_i\in y^{i-1}+(x)$$
for $1\le i \le s$.
For $1\le i\le r+1$, we define
$$\cI_i = (Q_i, xQ_{i+1}, x^2Q_{i+2}, \ldots, x^{r+1-i}Q_{r+1}).$$
\begin{lem}
For any field $F$ and ring homomorphism $\phi\colon \cA_\beta\to F$, $I_1 := \cI_1\otimes_{\cA_\beta,\phi} F$ is an $(x,y)$-primary ideal of $F[x,y]$ of codimension $n$.
A linear complement for $I_1$ in $F[x,y]$ is given by 
$$\Span\{x^{i-1} y^j\mid 1\le i \le r,\; 0\le j < \lambda_i\}.$$
Moreover, every $(x,y)$-primary ideal of $F[x,y]$ of codimension $n$
satisfying 
$$\dim (I_1: x^k)/(I_1:x^{k-1}) = \lambda_k,\ k=1,\ldots,r$$
arises from one and only one $\phi$.
\end{lem}

\begin{proof}
Setting $I_k := \cI_k\otimes_{\cA,\phi} F$, we have
$$I_k = (\bar Q_k, x \bar Q_{k+1},\ldots, x^{r+1-k}\bar Q_{r+1}),$$
where $\bar Q_k := Q_k\otimes 1$ belongs to $y^{\lambda_k} + (x)\subset F[x,y]$.  
As
$$\bar Q_{k} = y^{\lambda_{k}-\lambda_{k+1}} \bar Q_{k+1} + \sum_{j=1}^{\lambda_{k+1}} a_{kj} x^{\beta_j-k} y^{j-\lambda_{\beta_j+1}-1}\bar Q_{\beta_j+1},$$
where $a_{kj} := x_{kj}\otimes 1 = \phi(x_{kj})$, we have
$$x \bar Q_k \in (x\bar Q_{k+1},\ldots,x^{r+1-k} \bar Q_{r+1}) = x I_{k+1},$$
so $R \bar Q_k\in x I_{k+1}$ if and only if $R\in (x)$.  This means an element of $I_k$ belongs to $(x)$ if and only if it belongs to $xI_{k+1}\subset I_k$, i.e.
$$(I_k:x) = I_{k+1}$$
for $1\le k \le r$.  By induction,
$(I_k:x^j) = I_{k+j}$ for $1\le k < k+j \le r+1$.
As $I_{r+1}$ is the unit ideal, $x^r\in I_1$, so the image of $x$ in $F[x,y]/I_1$ is nilpotent.  As $y^{\lambda_1}$ is divisible by $x$ (mod $I_1$), it follows that $y$ is
nilpotent in $F[x,y]/I_1$.  Thus $I_1$ is $(x,y)$-primary.

The composition of maps $I_k\hookrightarrow F[x,y] \twoheadrightarrow F[y]$
sends $x^i \bar Q_{k+i}$ to $0$ for $i>0$ and sends $\bar Q_k$ to $y^{\lambda_k}$.
Thus, we have an isomorphism
\begin{equation}
\label{slice}
F[x,y]/(I_k+(x)) \isom F[y]/(y^{\lambda_k}).
\end{equation}

We prove by descending induction that the span of
\begin{equation}
\label{complement}
\{x^{i-k} y^j\mid k\le i \le  r,\;  0\le j < \lambda_i\}
\end{equation}
is complementary to $I_k$ in $F[x,y]$.
This is trivial for $k=r+1$.  
Multiplication by $x$ gives an isomorphism  
$$F[x,y]/I_{k+1}= F[x,y]/(I_k:x)\to (x)/(I_k\cap (x)).$$
By (\ref{slice}), the short exact sequence 
$$0\to (x)/(I_k\cap (x)) \to F[x,y]/I_k \to F[x,y]/(I_k+(x))\to 0$$
can be rewritten
$$0 \to F[x,y]/I_{k+1} \to F[x,y]/I_k \to F[y]/(y^{\lambda_k})\to 0.$$
By induction, 
$$\dim F[x,y]/I_k = \lambda_k+\lambda_{k+1}+\cdots+\lambda_r.$$
To prove the image of
(\ref{complement}) spans $F[x,y]/I_k$, we assume the corresponding statement for $k+1$.
Then
$$\{x^{i-k} y^j\mid k+1\le i \le  r,\;  0\le j < \lambda_i\}$$
spans $(x)/(I_{k+1}\cap(x))$ and the image of 
$$\{y^j\mid  0\le j < \lambda_k\}$$
spans $F[y]/(y^{\lambda_k})$, so the image of
(\ref{complement}) spans $F[x,y]/I_k$.

Next we claim that $I_1$ determines $\phi$.  Equivalently, $I_1$ determines $a_{ij} = \phi(x_{ij})$.
We prove by descending induction that $I_k$ determines $a_{ij}$ for all $i\ge k$.  This is trivial for $k\ge r+1$.
Assume it holds for $k+1$.  As $I_{k+1} = (I_k:x)$ determines $a_{ij}$ for $i\ge k+1$ (and therefore determines $\bar Q_{k+1},\ldots,\bar Q_{r+1}$), 
we need only consider the case $i=k$.
It suffices to prove that 
$$I_k\cap \Span\{x^{\beta_j-k}\bar P_j\mid 1\le j < \lambda_{k+1}\} = \{0\}.$$
Indeed, if 
$$\sum_{j=1}^{\lambda_{k+1}-1} c_j x^{\beta_j-k} \bar P_j\in I_k$$
and $m := \min\{\beta_j\mid c_j \neq 0\}$,
then this linear combination
lies in $I_k\cap (x^{m-k}) = x^{m-k}I_m$, and we have
$$\sum_{j=1}^{\lambda_m} c_j x^{\beta_j-m}\bar P_j \in I_m.$$
Reducing (mod $x$), we have a non-trivial linear combination of $y^{j-1}$ for $j \le \lambda_{\beta_j}\le  \lambda_m$ belonging to $(y^{\lambda_m})$, which is impossible.

Finally, we claim that every $(x,y)$-primary codimension-$n$ ideal of in $F[x,y]$ can be expressed as $I_1$ for some partition $\lambda$ of $n$ and some $\phi$.
Defining
$$\lambda_i = \dim (I:x^i)/(I:x^{i-1}),$$
we have $\lambda_1\ge \lambda_2\ge\cdots$ since multiplication by $x$ defines an injection
$$(I:x^{i+1})/(I:x^i)\hookrightarrow (I:x^i)/(I:x^{i-1}),\ i\ge1,$$
and $\sum_{i=1}^\infty \lambda_i = n$ since $(I:x^m)=F[x,y]$ for $m$ sufficiently large.  This determines $\lambda$,
and now we must show that the parameters $a_{ij}$ can be chosen so that $I_1 = I$.  We use induction on the number of parts in the partition.

Given $I$ with associated partition $\lambda_1\ge \lambda_2\ge \cdots$, let $J := (I:x)$, which is associated to $\lambda_2\ge \lambda_3\ge \cdots$.
By the induction hypothesis, there exist $a_{ij}\in F$ for $2\le i < r$, $1\le j\le \lambda_{i+1}$ such that 
$$I_2 = (\bar Q_2, x\bar Q_3,\ldots,x^{r-1}\bar Q_{r+1})$$
coincides with $J$.
The image of $I$ by the  (mod $(x)$) reduction map $F[x,y]\to F[y]$ is 
$(y^{\lambda_1})$, so $I =  (\bar Q_1)+xI_2$ for some $\bar Q_1$ of the form $y^{\lambda_1-\lambda_2}\bar Q_2 + x\alpha$, where
$$x\alpha \in (x) \cap J = (x)\cap I_2 = xI_3,$$
i.e., $\alpha\in I_3$.  On the other hand, if $\alpha - \beta \in I_2$, then
$$(y^{\lambda_1-\lambda_2} \bar Q_2 + x\alpha) + xI_2 = (y^{\lambda_1-\lambda_2} \bar Q_2 + x\beta) + xI_2.$$
It suffice to prove that every class in $I_3/I_2$ is represented by some $\alpha$ of the form $\sum_{i=1}^{\lambda_2}a_{1j}x^{\beta_j-2} y^{j-\lambda_{\beta_j}-1}\bar Q_{\beta_j+1}$.
Composing the map $F^{\lambda_2}\to I_3$ given by 
$$(a_{11},\ldots,a_{1\lambda_2})\mapsto \sum_{j=1}^{\lambda_2} a_{1j}x^{\beta_j-2} y^{j-\lambda_{\beta_j}-1}\bar Q_{\beta_j+1}\in I_3$$
with the quotient map $I_3\twoheadrightarrow I_3/I_2$, we get an injective map between vector spaces of dimension $\lambda_2$, which must therefore be surjective.

\end{proof}

Now we can prove Proposition~\ref{ES}.

\begin{proof}
It therefore suffices to prove the equivalent form
$$[R_n] = \sum_{\lambda\in P(n)} [\A^{n-\lambda_1}].$$

As $\cI_1$ contains $(x,y)^n$, if $M$ denotes $\cA_\beta$-module of polynomials of degree $< n$,
we have an isomorphism of $\cA_\beta$-modules $M/M\cap \cI_1\isom \cA_\beta/\cI_1$.  The $\cA_\beta$-linear map
$$\Span_{\cA_\beta}\{x^i y^j\mid 0\le i < r,\; 0\le j < \lambda_{i+1}\} \to  \cA_\beta/\cI_1$$
becomes an isomorphism after tensoring by any residue field of $\cA_\beta$, so by Nakayama's lemma, it must be an isomorphism.
Thus $\cA_\beta[x,y]/\cI_1$ is a free $\cA_\beta$-module, and this remains true after tensoring over $\Z$ with $k$.
If $S=\Spec \cA_\beta\otimes_{\Z} k$ and $Z = \Spec \cA_\beta[x,y]/\cI_1 \otimes_{\Z} k$, then $Z\to S$ is flat and therefore defines an $S$-point of
the Hilbert scheme $(\A^2)^{[n]}$, and since every geometric point of $S$ corresponds to a $(x,y)$-primary ideal, it follows that
$S$ maps to $R_n$.  At the level of $F$-points, this map gives a bijection between $(x,y)$-primary ideals associated to $\lambda$ and $F$-points of $S$.
The proposition now follows from the following lemma.
\end{proof}

\begin{lem}
Let $k$ be a field and $\phi\colon Y\to X$ a morphism of $k$-varieties.  If for all extension fields 
$F$ of $k$,  $\phi$ defines a bijection from $Y(F)$ to $X(F)$, then $[X] = [Y]$ in $K_0(\Var_k)$.
\end{lem}

\begin{proof}
Suppose $K$ is a field and $Y_K$ a $K$-variety such that for every extension field $L$ of $K$, there is a unique morphism $\Spec L\to Y_K$ lifting $\Spec L\to \Spec K$.
If $y_1,y_2$ are points on $Y_K$ with residue fields $K_1$ and $K_2$ over $K$, we can choose a field $\Omega$ in which $K_1$ and $K_2$ both embed as subfields,
so $Y_K$ has at least two distinct $\Omega$-points, contrary to assumption.  Thus $Y_K$ has a single point, so it is affine: $Y_K = \Spec A_K$ for some reduced $K$-algebra
$A_K$.  The nilradical corresponds to the unique maximal ideal, and it is zero since $Y_K$ is a variety, so $A_K$ is a field extension of $L$.  On the other hand, the identity map $\Spec K\to \Spec K$ lifts to $\Spec K\to \Spec A_K$, so the extension $K\to A_K$ has an inverse, which means it is trivial.

We apply this in the case that $K$ is the residue field of the generic point $\eta$ of a component of $X$ and $Y_K$ is the fiber of $Y$ over $\eta$.  The conclusion is that there exists a point $\eta'$ in $Y$ over $\eta$ for which $\phi$ gives an isomorphism of residue fields.  Thus, there exist open neighborhoods $U$ of $\eta$ in $X$ and $U'$ of $\eta'$ in $Y$ such that
$\phi^{-1} (U) = U'$  and $\phi$ induces an isomorphism $U'\to U$.  Replacing $Y$ and $X$ by $Y\setminus U'$ and $X\setminus U$ respectively, the restriction of $\phi$ induces a map on $F$-points for all extensions $F$ of $k$, and the lemma follows by Noetherian induction.

\end{proof}

\section{Chow motives and finite Galois modules}

Fix a field $k$ and denote by $V(k)$ the category of smooth and projective $k$-varieties and arbitrary morphisms of such varieties. Given $X\in V(k)$ of dimension $d$ we consider the graded Chow ring $A^*(X)=\oplus _{r=0}^{d}A^{d-r}(X)$ of cycles on $X$ modulo rational equivalence, where the group $A^{d-r}(X)=A_r(X)$ consists of classes of cycles of dimension $r$ \cite{Fulton}.
Let us recall a version of the category of Chow motives that is appropriate for our needs. First consider the additive category $\Cor(k)$ whose objects are the objects of $V(k)$ and morphisms are the degree zero Chow correspondences. That is given $X,Y\in \Cor(k)$, $X$ being of pure dimension $d$, we set
$$\Hom _{\Cor(k)}(X,Y):=A^d(X\times Y)$$
The composition of morphisms is the composition of correspondences \cite{Manin}. The category $\Cor(k)$ is the ``additivization" of the category $V(k)$. Next one defines the category $\Chow(k)$ of Chow motives as the idempotent completion of $\Cor(k)$. Explicitly the objects of $\Chow(k)$ are pairs
$(X,p)$, where $X\in V(k)$ and $p\in \End _{\Cor(k)}(X)$ is a projector: $p^2=p$. Morphisms
between $(X,p)$ and $(Y,q)$ form the group $q\cdot \Hom _{\Cor(k)}(X,Y)\cdot p$. There is a canonical contravariant functor
$V(k)\to \Chow(k)$ which sends $X\in V(k)$ to $(X,\bf{1})$ and a morphism $f:X\to Y$ to its graph $\Gamma _f\subset  Y\times X$.
Let $e\in \Chow(k)$ be the image of $\Spec k$.
The category $\Chow(k)$ is a tensor category with the product
$$(X,p)\otimes (Y,q)=(X\times Y,p\otimes q)$$
There exists an object ${\bf L}\in \Chow(k)$, called the Tate motive, such that $\bbP ^1=e\oplus \bf{L}$ \cite{Manin}.
For $(X,p)\in \Chow(k)$ we denote as usual the product $(X,p)\otimes {\bf L}$ by $(X,p)(-1)$
.
Given a nonzero integer $n$ we denote by $\Chow(k)[1/n]$ the localization at $n$ of the additive category $\Chow(k)$, i.e. for $A,B\in \Chow(k)$ we have
$$\Hom _{\Chow(k)[1/n]} (A,B)=\Hom _{\Chow(k)}(A,B)\otimes _\bbZ \bbZ[1/n]$$
So $\Chow(k)[1/n]$ is a $\bbZ[1/n]$-linear tensor category. We also consider the category $\Chow(k)_{\bbQ}$ of rational Chow motives constructed in a similar way.

\begin{example} \label{ex-proj-gr} Let $X\in V(k)$ be a variety of pure dimension $d$ with an action of a finite group $G$ of order $n$. Then $p:=\frac{1}{n}\sum _{g\in G}\Gamma _g\in A^d(X\times X)\otimes_{\bbZ} \bbZ[1/n]$ is a projector. Hence $(X,p)\in \Chow(k)[1/n]$.
\end{example}

Given a field extension $k\subset k^\prime$ we obtain the obvious functors $V(k)\to V(k^\prime )$, $\Cor(k)\to \Cor(k^\prime )$, $\Chow(k)\to \Chow(k^\prime )$, etc. induced by the extension of scalars $X\mapsto X_{k^\prime}=X\times_kk^\prime$ of varieties \cite[Example 6.2.9]{Fulton}.
If $k^\prime =\bar{k}$, as usual, we denote the variety $X\times_k\overline{k}$ by $\bar{X}$. For a prime $l\neq \mathrm{char}(k)$ let $\zeta _l\in \bar{k}$ be an $l$-th root of 1.

\begin{prop} \label{prop-exist-real} Let $n$ be a nonzero integer and $l$ be a prime number not dividing $n$ and different from the characteristic of the base field k. Then the assignment
$$X\mapsto H^\bullet _{\et}(\bar{X},\one),\quad X\in V(k)$$
extends to a tensor (contravariant) functor from the category $\Chow(k)[1/n]$ to the abelian tensor category of  finite dimensional $\one$-modules with a continuous $\Gal_k$-action:
$$\Phi _l:\Chow(k)[1/n]\to \one\text{-}\Gal_k\hmod$$

If $k$ contains $\zeta _l$, then the module $\Phi _l({\bf L})$ is a 1-dimensional trivial $\one\text{-}\Gal_k$-module placed in degree 2.
\end{prop}

We do not claim originality for this proposition, but for lack of a reference, we provide a proof.

\begin{proof} Since the category $\one\text{-}\Gal_k\hmod$ is Karoubian and its localization $(\one\text{-}\Gal_k\hmod)[1/n]$ is equivalent to
 $\one\text{-}\Gal_k\hmod$, it suffices to construct a functor
from the additive category $\Cor(k)$ to $\one\text{-}\Gal_k\hmod$. We construct this functor as the composition of the extension of scalars functor $\Cor(k)\to \Cor(\bar{k})$ with a functor
$$\Psi _l:\Cor (\bar{k})\to \one\hvect$$
where $\one\hvect$ is the category of  $\one$-vector spaces. The functor $\Psi _l$ is defined as follows. Let $X$ and $Y$ be smooth projective varieties (over $\bar{k}$), $X$ being of pure dimension $d$, and let $C\in A^d(X\times Y)$ be a correspondence of degree zero. Consider the projections $X\stackrel{p_X}{\leftarrow} X\times Y\stackrel{p_Y}{\rightarrow} Y$.  Then given an element $a\in H^i_{\et}(Y,\one)$ we put
$$\Psi _l(C)(a)=p_{X*}(\cycl_{X\times Y}(C)\cup p^*_Y(a))\in H^i_{\et}(X,\one)$$
where $\cycl_{X\times Y}:A^s(X\times Y)\to H^{2s}_{\et}(X\times Y,\one)$ is the {\it cycle map} \cite{SGA45}, \cite[Ch.VI, 9]{Mi}, and $p^*_Y$ and $p_{X*}$ are the pullback and the pushforward maps on cohomology \cite[Ch. VI, Remark 11.6]{Mi}. In order for $\Psi _l$ to be a functor, the cycle map has to satisfy the following properties for morphisms of smooth and projective varieties:
\begin{itemize}
\item $\cycl$ is a morphism of contravariant functors from $V(\bar{k})$ to the category of  rings.
\item $\cycl$ commutes with exterior products.
\item $\cycl$ is a morphism of covariant functors from $V(\bar{k})$ to the category of  abelian groups.
\end{itemize}

The first two properties are proved in \cite[2.3.9 and 2.3.8.3]{SGA45}, and the last one is in \cite[Theorem 6.1]{Laumon}.

Once the functor $\Psi _l$ is constructed, it is clear that its composition with the extension of scalars $\Cor(k)\to \Cor(\bar{k})$ will give the desired functor $\Phi _l$, since for $X\in \Cor(k)$ the vector space $H^\bullet _{\et}(\bar{X},\one)$ is a $\Gal_k$-module and morphisms in $\Cor(k)$ act as morphisms of $\Gal_k$-modules. Also the last assertion of the proposition is obvious.  This proves Proposition~\ref{prop-exist-real}.

\end{proof}

\begin{example} \label{example-quot} Let $(X,p)\in \Chow(k)[1/n]$ be as in Example \ref{ex-proj-gr} and let $l$ be prime to $n$ and $l\neq \mathrm{char}(k)$. Then $\Phi _l((X,p))=H^\bullet _{\et}(\bar X,\one)^G$ as $\one\text{-}\Gal _{k}$-modules.
\end{example}

\begin{cor}\label{cor-from-real} Assume that in $\Chow(k)_{\bbQ}$ we have an isomorphism of objects $A\simeq B$. Then for a divisible enough integer $n$ the objects $A$ and $B$ belong to the essential image of the category $\Chow(k)[1/n]$ and are isomorphic in $\Chow(k)[1/n]$. Fix one such $n$ and let $l$ be a prime not dividing $n$ and $l\neq \mathrm{char}(k)$. Then the $\one\text{-}\Gal _{k}$-modules $\Phi _l(A)$ and $\Phi _l(B)$ are defined and are isomorphic.
\end{cor}

\begin{proof} Indeed, an isomorphism in $\Chow(k)_{\bbQ}$ between $A$ and $B$ is witnessed
by a finite diagram of objects and correspondences with denominators. Hence this diagram exists in $\Chow(k)[1/n]$ for a divisible enough $n$. So $A\simeq B$  in such category $\Chow(k)[1/n]$. The last assertion now follows from Proposition \ref{prop-exist-real}.
\end{proof}

The category $\Chow(k)_{\bbQ}$ can be extended to include objects $(X^\prime ,p^\prime)$ where $X^\prime $ is a quotient variety
under the action of a finite group on a smooth projective variety, and $p^\prime \in A(X^\prime \times X^\prime)_{\bbQ}$ is a projector \cite[Example 16.1.13]{Fulton}. Denote the resulting category by $\Chow(k)_{\bbQ}^\prime$. The following lemma is proved in \cite[1.2]{dBN}.

\begin{lemma} \label{lemma-aznar}1) The obvious functor $\Chow(k)_{\bbQ}\to \Chow(k)_{\bbQ}^\prime$ is an equivalence of categories.

2) Let $X$ be a smooth projective variety with an action of a finite group $G$. Consider the motive $(X,p)\in \Chow(k)_{\bbQ}$ as in Example \ref{ex-proj-gr}. Then the motives $(X,p)$ and $(X/G,{\bf 1})\in \Chow (k)^\prime _{\bbQ}$ are isomorphic.
\end{lemma}

We remark that the measures $\mtot$ and $\ntot$ defined in \S4 factor through $K_0(\Chow(k)[1/n])$ if $\ell\nmid n$.
Indeed, by a theorem of Gillet and Soul\'e \cite[Theorem 4]{GS}, the correspondence $X\mapsto (X,{\bf 1})$ for a smooth and projective $X$ extends to a group homomorphism $\theta :K_0(\Var_k)\to K_0(\Chow(k))$, where $K_0(\Chow(k))$ is the Grothendieck group of the additive category $\Chow(k)$. Denote by $\theta [1/n]$ the composition of $\theta$ with the obvious homomorphism
$K_0(\Chow(k))\to K_0(\Chow (k))[1/n]$. 

Assume that the base field $k$ is a subfield of $\bbC$. The additive functor $\Phi _l$ of Proposition  \ref{prop-exist-real} 
induces the group homomorphism 
$$K_0(\Phi _l):K_0(\Chow (k)[1/n])\to \K(\one\text{-}\Gal_k)$$
such that we have the equality
\begin{equation}\label{equality} 
\mtot=K_0(\Phi _l)\circ \theta [1/n]:K_0(\Var_k)\to \K(\one\text{-}\Gal_k)
\end{equation}
and hence also 
$$\ntot=\Res _{\Gal _{k(\zeta _l)}}^{\Gal_k}\circ K_0(\Phi _l)\circ \theta [1/n]:K_0(\Var_k)\to \K(\one\text{-}\Gal _{k(\zeta _l)})$$

We obtain the following important corollary which is used in the proof of our main Theorem~\ref{main} below. 

\begin{cor} 
\label{bounded-quot}
Assume that $k\subset \bbC$ and let $X$ be a smooth projective variety over $k$ with an action of a finite group $G$. Then for all primes $l$ 
sufficiently large in terms of $|G|$ we have 
$$\mtot([X/G])=[H^\bullet (\bar{X},\one)^G]\in \K(\one\text{-}\Gal _k)$$ and hence also  
$$\ntot([X/G])=[H^\bullet (\bar{X},\one)^G]\in \K(\one\text{-}\Gal _{k(\zeta _l)})$$
\end{cor}

\begin{proof} By \ref{equality} it suffices to prove that for some $m$ and all $l$ sufficiently large in terms of $|G|$ we have 
$$K_0(\Phi _l)\circ \theta [1/m]([X/G])=[H^\bullet (\bar{X},\one)^G]\in \K(\one\text{-}\Gal _k)$$
Let $\vert G\vert =n$ and let $(X,p)\in \Chow (k)[1/n]$ be as in Example \ref{ex-proj-gr}. By Lemma \ref{lemma-aznar} there is an isomorphism of objects $(X/G ,{\bf 1})\simeq (X,p)$ in the equivalent categories $\Chow ^\prime (k)_{\bbQ}\simeq \Chow (k)_\bbQ$. Hence the same isomorphism holds in $\Chow (k)[1/m]$ for a divisible enough $m$ (Corollary \ref{cor-from-real}). 
Fix one such $m$ and let $l$ be a prime not dividing $m$. Then
\begin{align*}
\mtot([X/G])=K_0(\Phi _l)\circ \theta [1/m]([X/G])&=K_0(\Phi _l)([(X,p)]) \\
&=[\Phi _l((X,p))]=[H^\bullet (\bar X,\one)^G]
\end{align*}
where the last equality is by Example \ref{example-quot}.
This proves the corollary.
\end{proof}

\section{The main theorem}

In this section, we prove the main result of this paper:

\begin{thm}
\label{main}
There exists a K3 surface $X/\Q$ such that 
$$\zeta_{X}(t)\in K_0[\Var_{\Q}][1/\L][[t]]$$ 
is irrational in the sense that 
if $B(t)$ is a polynomial with coefficients in $K_0[\Var_{\Q}][1/\L][t]$ and $B(0)=1$, then $B(t)\zeta_{X}(t)$ is not a polynomial.
\end{thm}

\begin{proof}
Choose $X$ to be the variety denoted $X'$ in Proposition~\ref{K3}.  As usual we let $X^{[i]}$ denote the $i$th Hilbert scheme of points on $X$.
We assume the theorem is not true and choose $B(t)$ with $B(0)=1$ such that $A(t) := B(t)\zeta_{X}(t)$ is a polynomial.  Let $n=\deg A+\deg B + 24$.  

We fix a prime $\ell$ sufficiently large  that:

\begin{enumerate}

\item The homomorphism (\ref{key-group}) is an isomorphism, i.e., $\Gal(K_1 K_2/\Q(\zeta_\ell))$ is isomorphic to $\SL_2(\F_\ell)^2$.
\item For all $i\le n$, we have $\ntot([\Sym^i X]) = [\Sym^i H^\bullet(\bar X^n,\one)]$.
\end{enumerate}
For large enough $\ell$, (1) holds by Theorem~\ref{K3}, and (2) holds by Corollary~\ref{Symm} and Corollary~\ref{bounded-quot}.

We set $G := \SL_2(\F_\ell)^2$, which we identify with $\Gal(K_1 K_2 /\Q(\zeta_\ell))$.  Thus, for $1\le i\le n$, we have
\begin{equation}
\label{X1}
a_i := \ntot([\Sym^i X]) = \Res^G_{G_{\Q(\zeta_\ell)}}[\Sym^i(\one^{20}\oplus V_1\boxtimes V_1)] \in \Res_{G_{\Q(\zeta_\ell)}}^G \K(G).
\end{equation}
Let $b_i := \ntot([X^{[i]}])$ for $i\ge 0$.  Thus $b_i$ is effective for all $i$, and by Theorem~\ref{Gottsche-Th},
$a_i$ and $b_i$ satisfy
the identities  
given by (\ref{prod}).  

Applying  Lemma~\ref{keyb} with $R = \K(G_{\Q(\zeta_\ell)})$,  which contains $\K(G)$ via the map 
$\Res_{G_{\Q(\zeta_\ell)}}^G$,
we deduce that $\ntot(X^{[\ell]})$ is not effective, which is
a contradiction.

\end{proof}

\end{document}